\documentclass{article}
\usepackage[top=50pt,bottom=60pt,left=70pt,right=70pt]{geometry}
\usepackage{amsmath}
\usepackage{amssymb}
\usepackage{amsthm}
\usepackage{hyperref}
\hypersetup{
  colorlinks   = true, 
  urlcolor     = blue, 
  linkcolor    = blue, 
  citecolor   = red 
}
\usepackage{cleveref}
\usepackage[
backend=biber,
style=alphabetic,
sorting=ynt
]{biblatex}
\usepackage{setspace}
 \newcommand{\sbgrp}[1]{\langle #1\rangle}
\newcommand{\sbmon}[1]{{#1}^*}
\newcommand{\word}[1]{{#1}^{<\omega}}
 \newcommand\sect[2]{{
  \left.\kern-\nulldelimiterspace 
  #1 
  \littletaller 
  \right|_{#2} 
  }}
 \newcommand{\littletaller}{\mathchoice{\vphantom{\big|}}{}{}{}}

 \DeclareMathOperator{\val}{val}
 \DeclareMathOperator{\SUBGRP}{SGRP}
 \DeclareMathOperator{\SUBMON}{SMON}
 \DeclareMathOperator{\RAT}{RAT}
 \newcommand{\N}{\mathbb{N}}
\newcommand{\Z}{\mathbb{Z}}

\newcommand{\mcG}{\mathcal{G}}

\newtheorem{Theorem}{Theorem}

\newtheorem{Proposition}[Theorem]{Proposition}

\newtheorem{Conjecture}{Conjecture}
\theoremstyle{remark}

\title{Is decidability of the Submonoid Membership Problem closed under finite extensions?}
\author{Doron Shafrir\\
\small{doron.abc@gmail.com}}
\date{}
\begin{document}
\maketitle
\begin{abstract}
    We show that the rational subset membership problem in $G$ can be reduced to the submonoid membership problem in $G{\times}H$ where $H$ is virtually Abelian. We use this to show that there is no algorithm reducing submonoid membership to a finite index subgroup uniformly for all virtually nilpotent groups. We also provide evidence towards the existence of a group $G$ with a subgroup $H<G$ of index 2, such that the submonoid membership problem is decidable in $H$ but not in $G$.
\end{abstract}
\section{Introduction}
Given a f.g.\ group $G$, we look at 3 natural decision problems. The subgroup membership problem $\SUBGRP(G)$: Given a finite $S\subseteq G$ and $g\in G$, is it true that $g\in\sbgrp{S}$? The submonoid membership problem $\SUBMON(G)$: given a finite $S\subseteq G$ and $g\in G$, is it true that $g\in\sbmon{S}$? The rational subset membership problem $\RAT(G)$: given a description of a rational subset $R\subseteq G$ and $g\in G$, is $g\in R$? 

Instead of the above uniform problems, one may fix the subgroup $H$ (or submonoid, or rational subset) of $G$, and define the decision problem for that specific subset, where the only input is $g\in G$ and we have to decide if $g\in H$. It is not known for any of the above problems whether there exists a f.g.\ group $G$ such that the uniform decision problem is undecidable, but for any fixed subset of the given type, the membership problem is decidable. 

Since any f.g.\ subgroup is a f.g.\ submonoid, and any f.g.\ submonoid is a rational subset, we have trivial reductions $\SUBGRP(G)\le_T\SUBMON(G)\le_T\RAT(G)$. There are known examples of groups where $\SUBGRP(G)$ is decidable but $\SUBMON(G)$ is not, such as $\Z\wr\Z$ \cite{lohrey2015rational} and $H_3(\Z)^n$ for large $n$ \cite{roman2023undecidability}. There are also known examples of f.g.\ groups with decidable $\SUBMON(G)$ and undecidable $\RAT(G)$: There is a fixed rational subset  $R\subset \Z_2\wr\Z^2$ with undecidable membership problem \cite{lohrey2011tilings}, but $\SUBMON(\Z_2\wr\Z^2)$ is decidable, as shown in the Bachelor Thesis \cite{felix2020} based on an unpublished draft of the author \cite{shafrir2018unpub}. Recently, Bodart showed that there is a nilpotent group of class 2 such that $\RAT(G)$ is undecidable but $\SUBMON(G)$ is decidable \cite{bodart2024membership}, using a minimality argument. Our main argument elaborates on their argument.

If $H\le G$ is any subgroup, then trivially $\SUBGRP(H)\le_T\SUBGRP(G)$ and similarly for $\SUBMON$ and $\RAT$.  If additionally $[G:H]<\infty$, then $\SUBGRP(H)\equiv_T\SUBGRP(G)$ and $\RAT(H)\equiv_T\RAT(G)$ \cite{grunschlag1999algorithms}. However, no such equivalence is known for $\SUBMON$. It is a longstanding open problem whether there exists a group $G$ and a finite index subgroup $H$ such that the $\SUBMON(H)$ is decidable but  $\SUBMON(G)$ is not. We observe that in any such an example,  $\RAT(H)$ would be undecidable since $\RAT(H)\equiv_T\RAT(G)\ge_T\SUBMON(G)$, while $\SUBMON(H)$ is decidable, so finding such example is harder than the above problem.
The purpose of the current work is to provide evidence for the following conjecture:

\begin{Conjecture}\label{conj:main}
    There is a group $G$ with a subgroup $H$ of index $2$, such that $\SUBMON(H)$ is decidable but $\SUBMON(G)$ is undecidable.
\end{Conjecture}

The main result of this paper is the following dichotomy: Either \Cref{conj:main} is true, or there is a group $G$ such that membership in any fixed rational subset $R\subseteq G$ is decidable, but $\RAT(G)$ is undecidable. In fact, stronger conclusion can be made if \Cref{conj:main}  fails, see \Cref{thm:weirdland}. Our main tools are the reduction from \cite{bodart2024membership}, and reducing rational subset membership in $G$ to submonoid membership in $G{\times}H$ where $H$ is virtually Abelian, which will be shown in the next section.

\section{Rational sections of submonoids}
We will use the notation $\Z_n=\Z/n\Z$. Given a set $A\subseteq G\times H$ and $h\in H$, we define the $h$-section of $A$ as $\sect{A}{h}=\{g\in G\mid (g,h)\in A\}$.
\subsection{Product of 2 submonoids}
We first show that any product of 2 submonoids of a group $G$ is a section of a submonoid of $G{\times}H$ where $H$ is virtually Abelian. 
\begin{Proposition}
\label{prop:xn_s_yn}
There is a virtually Abelian group $H$ and elements $x,s,y\in H$ with the following property: A word $w\in\word{\{x,s,y\}}$ (without inverses) has value $s$ iff $w=x^lsy^l$ for some $l\in\N_0$.
\end{Proposition}
\begin{proof}
    Recall the notation $x^s=s^{-1}xs$. Define     \[H=\langle x,s,y\mid [x,y]=1, x^s=y^{-1}, y^s=x\rangle\]
Equivalently, $H=\Z^2\rtimes_s\Z$ where $x,y$ are a basis of $\Z^2$, and $s$ acts on $\Z^2$ by $x\mapsto-y,y\mapsto x$, that is, a rotation by $\pi/2$. In particular this action has order $4$; so the action of $\Z$ defining the semi-direct product is not faithful. We have $\sbgrp{x,y,s^4}\simeq\Z^3$ and $H/\sbgrp{x,y,s^4}\simeq \Z_4$, so $H$ is $\Z^3$-by-$\Z_4$. 
    If a word $w$ has value $s$, then $s$ appears exactly once in $w$, as can be seen by projecting $H\rightarrow H/\Z^2\simeq\Z$. Therefore, the value of $w$ can be expressed as:
    \[(x^ay^b)s(x^cy^d)=s(y^{-a}x^b)(x^cy^d)=sx^{b+c}y^{d-a}=s\]
    Where $a,b,c,d\ge0$.  We obtain $a=d,b=c=0$ so $w$ has the required form.
\end{proof}
The proof of \Cref{prop:xn_s_yn} clarifies why we defined $H=\Z^2\rtimes_s\Z$  rather than with the faithful action $H=\Z^2\rtimes_s\Z_4$: in the latter we cannot guarantee $s$ appears once. 
There are many other virtually Abelian groups $H$ satisfying \Cref{prop:xn_s_yn}:
For any $n\ge 3$ we can take $H=\Z^n\rtimes_s\Z$ where $s$ is the cyclic shift, i.e.\ $x_i^s=x_{i+1}$ for $i\in\Z_n$, and $x=e_0^{-1},y=e_1$. \Cref{prop:xn_s_yn} follows since $(e_0^{-a}e_1^b)s(e_0^{-c}e_1^d)=s(e_1^{-a}e_2^b)(e_0^{-c}e_1^d)=s$ iff $a=d,b=c=0$. $H$ is $\Z^{n+1}$-by-$\Z_n$. Thus for every prime $p>2$, $H$ can be chosen Abelian-by-$\Z_p$.

The property of \Cref{prop:xn_s_yn} directly gives us products of 2 submonoids as sections:

\begin{Proposition}
\label{prop:reduce2pair}
    Let $H$ be a group satisfying \Cref{prop:xn_s_yn}.  For any group $G$ and f.g.\ submonoids $A,B\le G$, there is a f.g.\ submonoid $C\le G{\times}H$ such that $AB$ is a section of $C$.
\end{Proposition}
\begin{proof}
    Let $A_0,B_0\subseteq G$ be finite sets such that $A^*_0=A,B^*_0=B$, WLOG $1_G\in A_0,1_G\in B_0$. Let $x,s,y\in H$ as in \Cref{prop:xn_s_yn}. Define:
    \[ C_0 = A_0{\times}\{x\} \cup \{(1_G,s)\} \cup B_0{\times}\{y\} \]
    and $C=C^*_0$. We claim that $\sect C s=AB$.  We first prove $AB\subseteq\sect C s$.
    Let  $g=a_1\cdots a_n\cdot b_1\cdots b_m\in AB$ where $a_i\in A_0,b_i\in B_0$. Since $1_G\in A_0, 1_G\in B_0$ we can assume $m=n$ by padding the shorter sequence with $1_G$. We have:
\begin{equation}\label{eq:prod_pair}
       (a_1,x)\cdots(a_n,x)(1_G,s)(b_1,y)\cdots(b_n,y)=(g,x^nsy^n)=(g,s)\in C
\end{equation}
    Proving $g\in\sect C s$. Since $g\in AB$ is arbitrary, $AB\subseteq \sect C s$. Conversely, let $g\in\sect C s$, so $(g,s)=c_1\cdots c_k$ for $c_i\in C_0$. By \Cref{prop:xn_s_yn} applied on the second coordinate, we get that $c_1\cdots c_k$ must have the form of \Cref{eq:prod_pair}, and therefore $g\in AB$.
\end{proof}

\subsection{General rational sections}
We now go from products of 2 submonoids to general rational subsets. 
We use the wreath product $\Z\wr\Z_n=\Z^n\rtimes_s\Z_n$. Define $e_i$ as the basis elements of $\Z^n$ for $i\in\Z_n$, and $s$ the shifting element, so $e_i^s=e_{i+1}$. By a slight abuse of notation, we write $s^i$ where the exponent $i$ is in $\Z_n$ rather than $\Z$, which is well-defined since $ord(s)=n$.
\begin{Proposition}\label{prop:path_words}
    Let $H=\Z\wr\Z_n$ and $t_{ij}=e_is^{j-i}\in H$ for $i,j\in\Z_n$. A word $w\in\word{\{t_{ij}\mid i,j\in\Z_n\}}$  has value  $e_0^ls\in H$ for some $l\in\N$ iff  $w=t_{i_0i_1}t_{i_1i_2}\cdots t_{i_{l-1}i_l}$ for some sequence  $(0=i_0,i_1,...,i_l=1)\in\Z_n^{l+1}$.
\end{Proposition}
If we think of $t_{ij}$ as corresponding to an edge $i\rightarrow j$, then words of the above type correspond to paths. We chose $0$ and $1$ as the starting point and end point of the path out of convenience. In general, paths $i\rightarrow j$ of length $l$ correspond to words with value $e_i^ls^{j-i}$.
\begin{proof}
    We have $t_{ij}=e_is^{j-i}=s^{-i}e_0s^j$. If  $(0=i_0,i_1,...,i_l=1)\in\Z_n^{l+1}$ we get:
\[t_{i_0i_1}\cdots t_{i_{l-1}i_l}=(s^{-i_0}e_0s^{i_1})(s^{-i_1}e_0s^{i_2})\cdots(s^{-i_{l-1}}e_0s^{i_l})=s^{-i_0}e_0^ls^{i_l}=e_0^ls\]
So any such word has the required value. Conversely, let $w$ be a word in $\word{\{t_{ij}\mid i,j\in\Z_n\}}$ with value $e_0^ls$. Each element of $H$ has a unique canonical form $e_0^{c_0}e_1^{c_1}\cdots e_{n-1}^{c_{n-1}}s^k$ for some $c_i\in\Z,k\in\Z_n$. To bring $w$ to this form, we pull all $e_i$'s to the beginning of the word using $s^je_i=e_{i-j}s^j$.  We observe that each letter of $w$ contributes one $e_i$ to the canonical form, and there are no negative powers of $e_i$ that may cause cancellation. Therefore, since  $\val(w)=e_0^ls$, we have $|w|=l$ and every $e_i$ must become $e_0$ when it reaches the beginning of the word. Set $w=t_{i_0j_0}\cdots t_{i_{l-1}j_{l-1}}$. We must have $i_0=0$ since $e_{i_0}$ is already in the beginning of $w$. We prove by induction on $k$ that $w_1...w_k=e_0^ks^{j_k}$ and  $i_k=j_{k-1}$. Assume the claim is true for $k-1$, for $k$ we get 
\[w_1\cdots w_k=(e_0^{k-1}s^{j_{k-1}})(e_{i_k}s^{j_k-i_k})=e_0^{k-1}e_{i_k-j_{k-1}}s^{j_k-(i_k-j_{k-1})}\]
Since we must have $e_{i_k-j_{k-1}}=e_0$, we get $i_k=j_{k-1}$ and $w_1\cdots w_k=e_0^ks^{j_k}$ as needed. Finally, the value of $w$ is   $e_0^ls^{j_{l-1}}$ so $j_{l-1}=1$. 
\end{proof}

\begin{Proposition}\label{prop:pair2rat}
    Let $R\subseteq G$ be a rational subset of a group $G$ defined by a $G$-labelled graph on $n$ vertices, and set $H=\Z\wr\Z_{n+1}$. There is a f.g. submonoid $M\le G{\times}H$ and $y\in G{\times}H$ such that $R$ is a section of  $\sbmon{y}M$.
\end{Proposition}
\begin{proof}
    Let $V$ be a finite set of states of size $n$, and $S\subseteq V{\times}G{\times}V$ the set of transitions defining $R$. In order to have a single accepting state, we add $1_G$-labelled transitions from each accepting state to a new vertex, which will become the unique accepting state. We now have $n+1$ states, which we rename so that $V=\{v_0,...,v_n\}$, $v_0$ is the initial state, and $v_1$ is the accepting state. Let $H=\Z\wr\Z_{n+1}$, $t_{ij}$ as in \Cref{prop:path_words} and define:
\begin{equation*}
        T=\{(g,t_{ij})\mid(v_i,g,v_j)\in S\}
\end{equation*}
and $y=(1_G,e_0^{-1})$. We claim $R=\sect{(\sbmon{y}\sbmon{T})}{s}$.
First we prove $R\subseteq\sect{(\sbmon{y}\sbmon{T})}{s}$. Let $g\in R$, then $g$ is the label of an $S$-path of length $l$ from $v_0$ to $v_1$. Taking the product in $\sbmon{T}$ corresponding to this path we get $(g,e_0^ls)\in \sbmon{T}$ by \Cref{prop:path_words}, therefore $g\in\sect{(\sbmon{y}\sbmon{T})}{s}$.

Conversely, if $g\in \sect{(\sbmon{y}\sbmon{T})} s$, then $(g,s)\in\sbmon{y}\sbmon{T}$ so $(g,e_0^{l'}s)\in\sbmon{T}$ for some $l'\ge 0$. Let $u_1\cdots u_l$ be the $T$-product with this value. 
Applying \Cref{prop:path_words} to the second coordinate of this product, we conclude that $l=l'$ and that the second coordinate of each $u_k$ is $t_{i_{k-1}i_k}\in H$ for some sequence $(0=i_0,...,i_l=1)$. Looking at the first coordinates of $u_k$, we get an $S$-path from the initial state $0$ to the accepting state $1$, showing $g\in R$.
\end{proof}

We are now ready to prove the main tool of this paper.
\begin{Theorem}\label{thm:mon2rat}
 Let $R\subseteq G$ be a rational subset of a group $G$. There is a virtually Abelian group $H$ and a f.g. submonoid $M\le G{\times}H$ such that $R$ is a section of $M$. Moreover, $H$ depends only on the number of vertices defining $R$.
\end{Theorem}
\begin{proof}
    This is just a combination of \Cref{prop:reduce2pair} and \Cref{prop:pair2rat}.
    First by \Cref{prop:pair2rat} we get $H_1,h_1$ and $A,B\le G{\times}H_1$ such that $(AB)_{h_1}=R$ where $A$ is cyclic. Then by \Cref{prop:reduce2pair} we get $H_2,h_2$ and $M\le G{\times}H_1{\times}H_2$ such that $M_{h_2}=AB$.
    Set $H=H_1{\times}H_2,h=(h_1,h_2)$. We get $M_h=\sect {(\sect M {h_2})} {h_1}=\sect{(AB)}{h_1}=R$. the claim follows since $H_1$ depends only on the number of vertices defining $R$,  $H_2$ is a fixed group, and both are virtually Abelian.
\end{proof}
If $R\subseteq G$ is a rational set defined by a graph on $\le n$ vertices with a single accepting state, we proved that $R$ is a section of a finite extension of $G{\times}\Z^{n+3}$. Although we will not need it, we will show more efficient construction, where any such $R$ defined on $\le n!2^n$ vertices is a section of a finite extension of $G{\times}\Z^{n+3}$. Since it will not affect our main results, the reader may skip to the next section.

We note that if $A\le Aut(\Z^n)=GL_n(\Z)$ is finite, there is $v\in\Z^n$ such that $stab_A(v)=\{Id\}$, or equivalently, $|Av|=|A|$. This is because $rank\{v\in\Z^n\mid  av=v\}<n$ for any $a\in A\setminus\{Id\}$.
\begin{Proposition}\label{prop:autZn}
    Let $A\le Aut(\Z^n)$ be finite, and let $v\in\Z^n$ have a trivial stabilizer in $A$. Define $t_{ab}=a^{-1}vb\in\Z^n\rtimes A$ for any $a,b\in A$. A word $w\in\{t_{ab}\mid a,b\in A\}^l$ has value $v^la$ for some $a\in A$ iff $w=t_{a_0a_1}t_{a_1a_2}\ldots t_{a_{l-1}a_l}$ for some $(Id=a_0,a_1,\ldots ,a_l=a)\in A^{l+1}$.  
\end{Proposition}
\begin{proof}
   It is easy to see that any word in this form has the right value. Conversely, let $w$ be a word of length $l$ with  $\val(w)=v^la$. If $w$ is any word in the $t_{ab}$'s, we can pull all $v$'s to the left, and obtain $\val(w)=v_1\cdots v_la'$ for some $v_i\in Av$ and $a'\in A$. Since $\val(w)=v^la$ we obtain $a'=a$ and $\sum_{i\le l} v_i=lv$. We take any inner product on $\Z^n$ and average over $A$, and get an $A$-invariant inner product $\langle\cdot,\cdot\rangle$ on $\Z^n$. By scaling we may assume $\langle v,v\rangle =1$. We have  $|v_i|=1$  from the $A$-invariance of the inner product. By Cauchy-Schwartz we get $\langle v_i,v\rangle\le|v_i||v|=1$, with equality iff $v_i=v$.  Therefore we get $\sum_{i\le l}\langle v_i,v\rangle\le l$, but on the other hand $\sum_{i\le l} v_i=lv$ implies $\sum_{i\le l}\langle v_i,v\rangle=\langle lv,v\rangle =l$. Therefore we get $v_i=v$ for all $i$. From this, the same argument as in \Cref{prop:path_words} shows that $w$ has the form of a path $Id\rightarrow a$ as needed.
\end{proof}
Geometrically, we showed that $v$ is not in the convex hull of $(A\setminus\{Id\})v=Av\setminus\{v\}$. As an example of a large finite subgroup $A\le Aut(\Z^n)$ one may take the group of signed permutation matrices of size $n!2^n$, where we can use $v=(1,2,...,n)$ as the vector with trivial stabilizer. 
\begin{Theorem}
    \label{thm:tighterAut}
    Let $R\subseteq G$ be a rational set defined by $n$ vertices and one accepting state, let $H$ be a group satisfying \Cref{prop:xn_s_yn}, and let $A\le Aut(\Z^n)$ a finite group, where $|A|\ge n$. Then $R$ is a section of $G{\times}(\Z^n{\rtimes}A){\times}H$. 
\end{Theorem}
\begin{proof}
    We use $A$ as a set of vertices defining $R$, with initial state $Id$ and a single accepting state $c\in A$, and $S\subseteq A{\times}G{\times}A$  the set of transitions. Let $x,s,y\in H$ as in \Cref{prop:xn_s_yn} and $v\in\Z^n$ with $stab_A(v)=\{Id\}$. Define:
    \[T=\{(g,a^{-1}vb,x)\mid(a,g,b)\in S\}\cup\{(1_G,c^{-1},s),(1_G,v^{-1},y)\}\]
We claim $\sect{\sbmon{T}}{(1,s)}=R$. Indeed, any $g\in R$ is given by a path of some length $l$, giving a product $(g,v^lc,x^l)$, and we have
\[(g,v^lc,x^l)(1,c^{-1},s)(1,v^{-1},y)^l=(g,1,x^lsy^l)=(g,1,s)\in\sbmon{T}\]

Conversely, if $(g,1,s)\in\sbmon{T}$, then in the last coordinate we must have $x^lsy^l$, which implies that in the second coordinate we have $\prod_{i\le l} (a^{-1}_ivb_i)c^{-1}v^{-l}=1$ or $\prod_{i\le l}a^{-1}_ivb_i=v^lc$. By \Cref{prop:autZn} we get $a_1=Id$, $a_i=b_{i-1}$ and $b_l=c$, so the first coordinate traces an $S$-path from $Id$ to $c$, implying $g\in R$.
\end{proof}
Recall that in the proof of \Cref{prop:reduce2pair} we added $1_G$ to $A_0,B_0$ to remove the limitation that the 2 products have the same lengths. In the proof of \Cref{prop:pair2rat} it turns out that we don't need those elements since both products always has the same length $l$, but they do not cause any problem. In contrast, the proof of \Cref{thm:tighterAut} would fail if we add $(1,1,y)$, since the implication  $\sum_{i\le l} v_i=l'v\rightarrow v_i=v$ may fail if $l'<l$. Geometrically, in \Cref{prop:path_words} $A$ is the group of cyclic shifts, and we use the fact that $e_0$ is not in the cone generated by $Ae_0\setminus\{e_0\}=\{e_1,...,e_n\}$, whereas in \Cref{prop:autZn} we only know that $v$ in not in the convex hull of $Av\setminus\{v\}$. This is why we can't use a product of pair of submonoids in \Cref{thm:tighterAut} and had to use \Cref{prop:xn_s_yn} directly.
\subsection{Limitations}
It is natural to ask if we can go further and find a virtually Abelian group $H$ such that any rational subset of $G$ is a section of $G{\times}H$, without the limitation on the number of states defining $R$. We explain why such a group $H$ is unlikely to exist. Assume $H$ has a normal subgroup $N\simeq\Z^d$ and $[H:N]=n$. Given a f.g.\ submonoid $M\le G{\times}H$, we can describe the section $\sect{M}{h}$ as follows: by using the standard technique proving that $\RAT(G{\times}H)\equiv_T\RAT(G{\times}\Z^d)$, we build a $G{\times}\Z^d$-labelled graph $\Gamma$ on $n$ vertices corresponding to the cosets of $H/N$. We can then define the rational subset $R\subseteq G{\times}\Z^d$  given by paths $N\rightarrow hN$, and we have $\sect M h=\sect R v$ for some $v\in\Z^d$. Equating the $\Z^d$ coordinate to $v$ gives linear constraints on the number of uses of each edge of $\Gamma$, and forbids using some edges of $\Gamma$, so they can be erased. 
It is instructive to build $\Gamma$ for the above proofs, and see how $\Gamma$ reconstructs a graph that is equivalent to the input graph. However, by its nature such a construction cannot define a rational subset on more than $n$ vertices.

We have shown that any rational subset of $G$ is a section of a f.g.\ submonoid $M\le G{\times}H$ where $H$ is virtually Abelian. The converse is not true, even if $M$ is a subgroup and $H=\Z$: if $F_2=\sbgrp{a,b}$ is the free group on 2 generators, and $M=\sbgrp{(a,0),(b,1)}\le F_2{\times}\Z$, then the section $\sect M 0$ is the subgroup of $F_2$ freely generated by $\sbgrp{a^{b^i}\mid i\in\Z}$. Since every rational set generates a f.g.\ subgroup (\cite[Corrolay 1]{gilman1987groups}), $\sect M 0$ is not rational.

\section{Decidability}
The dependence of $H$ on $R$ in \Cref{thm:mon2rat} gives rise to the following problem. Given $G$ and $n\in\N$, we define the \emph{bounded rational subset membership}, denoted $\RAT_{\le n}(G)$, as the following problem:  Given a description of a rational subset $R\subseteq G$ with at most $n$ vertices, and $g\in G$, decide if $g\in R$. We have $\RAT_{\le1}(G)\equiv_T\SUBMON(G)$ by definition. In terms of uniformity, this problem lies between the membership problem in a specific rational subset $R\subseteq G$ and the full $\RAT(G)$. 

\Cref{thm:mon2rat} easily implies the following 3 reductions, in increasing order of uniformity:
\begin{Theorem}\label{thm:3reductions}
Let $G$ be a fixed f.g.\ group.
\begin{enumerate}
    \item The membership problem in a fixed rational set $R\subseteq G$ can be reduced to the membership problem in a fixed submonoid $M\subseteq G{\times}H$ where $H$ is a virtually Abelian group.
    \item For any fixed $n$, $\RAT_{\le n}(G)\le_T\SUBMON(G{\times}H)$, where $H$ is a virtually Abelian group depending  on $n$.
    \item Let $\mcG=\{G{\times}H\mid H\text{ is virtually Abelian}\}$. The uniform rational subset membership problem in $\mcG$ is equivalent to the uniform submonoid membership problem in $\mcG$.
\end{enumerate}
\end{Theorem}
For the following results, we note that in all the above results we can always choose the virtually Abelian group $H$ to have an Abelian normal subgroup $N$ such that $H/N\simeq \Z_4{\times}\Z_{2^k}$ for some $k$. This is because we can always increase $n$ in \Cref{prop:pair2rat} to a power of $2$. We now prove our first theorem on the behaviour of $\SUBMON$ under finite extensions.
\begin{Theorem}
    There is no algorithm that takes as input a description of a virtually nilpotent group $G$, a subgroup $H<G$ of index 2, and an instance of the submonoid membership problem in $G$, and solves it using oracle access to the submonoid membership problem in $H$.
\end{Theorem}
\begin{proof}
    If such an algorithm existed, we could decide $\RAT(G)$ uniformly for all nilpotent groups, which would contradict e.g.\ \cite{roman2023undecidability}, as follows: Given an instance of $\RAT(G)$, reduce it to $\SUBMON(G{\times}H)$ where $H$ is Abelian by finite-Abelian-2-group by \Cref{thm:3reductions}. Then use the hypothetical algorithm recursively, to reduce to subgroups of index $2$ until we reach an instance of  $\SUBMON(G{\times}\Z^r)$. Then apply \cite[Theorem A]{bodart2024membership} to reduce to instances of $\RAT(K)$ for a nilpotent group $K$ satisfying $h([K,K])<h([G{\times}\Z^r,G{\times}\Z^r])=h([G,G])$, where $h(\cdot)$ is the Hirsch length, and repeat recursively. Eventually we get an instance of $\RAT(A)$ where $A$ is virtually Abelian, which is decidable.
\end{proof}

Let $N_{2,m}$ be the free nilpotent group of class $2$. For the next theorem we recall a result from \cite{bodart2024membership}: 
\begin{Theorem}
\label{thm:dim_gain}
\cite[Corollary 2.8]{bodart2024membership}
Fix $m,n\in\N$. Then $\SUBMON(N_{2,m}{\times}\Z^n)\le_T\bigcup_{G\in\mcG}\RAT(G)$  where $\mcG=\{N_{2,m'}{\times}\Z^{n'}\mid m'<m, \binom{m'}{2}+n'\le \binom{m}{2}+n\}$ is a finite class of groups.
\end{Theorem}

Our evidence for \Cref{conj:main} is the following:
\begin{Theorem}
\label{thm:weirdland}
    If \Cref{conj:main} is false, both of the following statements are true:
    \begin{enumerate}
        \item There is a group $G$ such that the $\RAT_{\le n}(G)$ is decidable for every $n$, but  $\RAT(G)$ is undecidable.
        \item There is a group $L$ such that for every $n$, $\SUBMON(L{\times}(\Z\wr\Z_{2^n}))$ is decidable, but there is no single algorithm taking $n$ as input deciding submonoid membership in all of them uniformly.
    \end{enumerate}
\end{Theorem}
\begin{proof}
As in \cite{bodart2024membership}, define $m,r$ be such that $\RAT(N_{2,m}{\times}\Z^r)$ is undecidable and $m$ is minimal, and set $G=N_{2,m}{\times}\Z^r$ (such $m,r$ exist by \cite{konig2016knapsack}).  Let $H$ be the $\Z^3$-by-$\Z_4$ group from the proof of \Cref{prop:xn_s_yn}. Fix $n$ and set $K=G{\times}H{\times}(\Z\wr\Z_{2^n})$. By \Cref{thm:3reductions}, $\RAT_{\le n}(G)\le_T\SUBMON(K)$ (since $n+1\le 2^n$). We note that $K$ has a normal subgroup $N\simeq N_{2,m}{\times}\Z^{r+2^n+3}$ such that $K/N\simeq \Z_4{\times}\Z_{2^n}$. By minimality of $m$ and \Cref{thm:dim_gain}, $\SUBMON(N)$ is decidable. We can easily build a series $N=N_0<N_1\cdots <N_{n+2}=K$ with $[N_i:N_{i-1}]=2$. Assuming \Cref{conj:main} is false, we get by induction that $\SUBMON(N_i)$ is decidable, therefore $\SUBMON(K)$ and $\RAT_{\le n}(G)$ are decidable, proving the first claim. 
Now, setting $L=G{\times}H$, we have shown that $\SUBMON(L{\times}(\Z\wr\Z_{2^n}))$ is decidable for all $n$, but if there was an algorithm deciding it uniformly (or even for infinitely many cases), we could apply the reduction $\RAT_{\le n}(G)\le_T\SUBMON(L{\times}(\Z\wr\Z_{2^n}))$ uniformly to decide $\RAT(G)$, contradiction.
\end{proof}

While the conclusions of \Cref{thm:weirdland} are very interesting, we do not see it as evidence to believe that they hold, but rather that they are a result of our inability to give a definite proof of \Cref{conj:main}. A likely way to settle \Cref{conj:main} is to find a group $G$ with a fixed undecidable rational subset, such that $\SUBMON(G{\times}\Z^r)$ is decidable for every $r$. In this case, \Cref{thm:3reductions} shows that for some virtually Abelian $H$, $G{\times}H$ has a fixed submonoid with undecidable membership problem, and it has a finite index subgroup with decidable uniform submonoid membership problem.

\printbibliography
\end{document}